\documentclass{article}

 \usepackage[preprint]{nips_2018}

\usepackage[utf8]{inputenc} 
\usepackage[T1]{fontenc}    
\usepackage{hyperref}       
\usepackage{url}            
\usepackage{booktabs}       
\usepackage{amsfonts}       
\usepackage{nicefrac}       
\usepackage{microtype}      

\usepackage{amsthm}
\usepackage{subfigure}
\usepackage{graphics} 
\usepackage{epsfig} 
\usepackage{mathptmx} 
\usepackage{times} 
\usepackage{amsmath} 
\usepackage{amssymb}  
\usepackage{makeidx}
\usepackage{float}
\usepackage{balance}
\usepackage{booktabs}
\usepackage{bbm}
\usepackage{mathrsfs}
\usepackage{enumerate}
\usepackage{multicol}
\usepackage{algorithmic}

\usepackage[ruled,linesnumbered]{algorithm2e}
\newtheorem{theorem}{Theorem}[section]
\newtheorem{lemma}{Lemma}[section]

\newtheorem{definition}{Definition}[section]
\newtheorem{prop}{Proposition}[section]

\newcommand{\rd}{{\mathrm d}}

\newcommand{\tr}{{\mathrm {Tr}}}

\newcommand{\rP}{{\mathrm{P}}}
\newcommand{\rQ}{{\mathrm{Q}}}

\newcommand{\rT}{{\mathrm{T}}}

\newcommand{\F}{\mathscr{F}}

\newcommand{\A}{\mathscr{A}}
\newcommand{\Pc}{\mathscr{P}}

\newcommand{\vm}{{\bf m}}
\newcommand{\vx}{{\bf x}}

\newcommand{\vu}{{\bf u}}

\newcommand{\vM}{{\bf M}}

\newcommand{\calA}{{\cal A}}

\newcommand{\calF}{{\cal F}}

\newcommand{\calO}{{\cal O}}
\newcommand{\calU}{{\cal U}}
\newcommand{\calL}{{\mathcal{L}}}

\newcommand{\vTheta}{{\mbox{\boldmath$\Theta$}}}

\newcommand{\vtheta}{{\mbox{\boldmath$\theta$}}}

\newcommand{\argmax}{\operatornamewithlimits{argmax}}
\newcommand{\argmin}{\operatornamewithlimits{argmin}}

\newcommand{\bbX}{\mathbb{X}}

\newcommand{\Rb}{\mathbb{R}}
\newcommand{\Eb}{\mathbb{E}}
\newcommand{\Qb}{\mathbb{Q}}
\newcommand{\Pb}{\mathbb{P}}

\title{\LARGE \bf
Variational Inference for Stochastic Control of Infinite Dimensional Systems
}

\author{George I. Boutselis$^*$, Marcus Pereira\thanks{Equal contribution}~ and Evangelos  A. Theodorou \\
Autonomous Control and Decision Systems Laboratory \\
School of Aerospace Engineering\\
Georgia Institute of Technology
  }

\begin{document}

\maketitle


\begin{abstract}
  
    This paper develops a variational inference framework  for  control of infinite dimensional stochastic systems. We employ a measure theoretic approach which relies on the generalization of Girsanov's theorem, as well as the relation between relative entropy and free energy. The derived control scheme is applicable to a large class of stochastic, infinite dimensional systems, and can be used for trajectory optimization and model predictive control. Our work opens up new research avenues at the intersection of stochastic control, inference and information theory for dynamical systems described by stochastic partial differential equations.

\end{abstract}

\section{Introduction}
In many practical applications, one faces the problem of controlling dynamical systems represented by stochastic partial differential equations (SPDEs). Examples can be found, for instance, in fluid mechanics, open quantum systems,  turbulence, plasma physics and partially observable stochastic control \cite{chow2007stochastic,da1992stochastic,StochasticNavierStokes_2004,Dumont_2017,Pardouxt_1980,PhysRevE_Bang1994,RAMA_SPDEs_Economics_2005,2017arXiv171009619K}.  Despite  the importance of such applications, the majority of works on computational stochastic control has been dedicated to finite dimensional systems.  These are systems  represented by stochastic differential equations (SDEs), and can be found in a plethora of applications from   robotics and  autonomous systems, to  computational neuroscience, biology and finance. In contrast, the literature is lacking works on scalable/implementable control schemes for stochastic, infinite dimensional systems.  To this end, this paper tries to bridge the gap between theory and implementation of stochastic control in infinite dimensions. Our approach is based on the free energy-relative entropy duality, and utilizes elements from  stochastic calculus in Hilbert spaces. The  resulting methodology avoids restictive assumptions about the problem formulation, and can be applied to a broad class of semilinear SPDEs.   
 
  Previous work in the area of control of SPDEs has focused on very  specific systems, and typically  consists of theoretical results on the existence and uniqueness of solutions. References \cite{DaPrato1999} and \cite{feng2006} share some common characteristics with our paper.  In particular, the former work investigates explicit solutions of the Hamilton-Jacobi-Bellman (HJB) equation for the stochastic Burgers equation.  The derivation is based on the exponential transformation of the value function, as well as the transformation of the backward HJB equation into a forward Kolmogorov equation. Then, the  explicit solution is recovered  through the forward Feynman-Kac lemma and a probabilistic representation of the value function. The work in \cite{feng2006}  extends the large deviation theory to infinite dimensional systems, and creates connections to HJB theory. The analysis therein shows that a   free energy-like function corresponds to the value function  of a deterministic optimal control problem under a specific cost functional. This connection is established  by proving  that the aforementioned free energy-like function satisfies the  HJB equation of an infinite horizon deterministic optimal control problem. 

 On the  computational side,  the  work in \cite{Christofides2009} proposes a  model predictive control methodology  for nonlinear dissipative SPDEs.  The key idea lies in model reduction; that is, the transformation of the original    SPDE  into a set of coupled stochastic differential equations. Once this finite dimensional representation is obtained, a  model predictive control methodology is developed and is then applied on the \textit{Kuramoto-Sivanshisky} equation. Another work on computational control of the aforementioned SPDE can be found in \cite{GOMES201733}. This approach shares similarities  with the one in  \cite{Christofides2009}, in that  a finite dimensional representation  of the SPDE is utilized, rendering thus the use of standard control theory feasible.   
 
To the best of our knowledge, the framework developed in this paper is the first  step towards  explicitly designing implementable, numerical stochastic  control  algorithms  in infinite dimensions. In contrast to prior work (see \cite{Christofides2009,GOMES201733}),  the proposed approach  treats SPDEs as time-indexed stochastic processes taking values  in an  infinite dimensional space. The core of our methodology relies on sampling stochastic paths from the dynamics, and computing the associated trajectory costs. Grounded on the theory of  stochastic calculus in Hilbert spaces, we are not restricted to any particular finite representation of the original system. Besides the theoretical implications, this fact is also benfecial from a computational standpoint. Specifically, the obtained expressions for our control updates are independent of the method used to actually simulate the SPDEs. This further implies that the required sampled paths can be obtained by employing the scheme that is more suitable to each particular problem setup (e.g., finite differences, Galerkin methods or finite elements). Finally, we note that this work can be considered as a generalization of the Path Integral and information theoretic control method \cite{todorov2009efficient,TheodorouCDC2012,entropy_2015,Kappen2005b}. As such, the proposed  stochastic control algorithm can be efficiently applied in a Model Predictive Control (MPC) fashion, and  inherits  the ability to deal with non-quadratic cost functions and nonlinear dynamics. 
 
   The rest of the paper is organized as follows: In section \ref{sec:Preliminaries} we provide some important definitions and theorems on infinite dimensional stochastic systems.  In section \ref{sec:Free_Energy_Relative_Entropy} we discuss the free energy and relative entropy relation. Based on this connection, section \ref{sec:Inference} derives our stochastic control method by performing  inference in Hilbert spaces. Furthermore, in subsection \ref{sec:Iterative} we develop an iterative version of our framework, which is subsequently tested in simulation in section \ref{sec:Experiments}. Section  \ref{sec:conclusions} concludes the paper.


  
\section{Preliminaries - Stochastic Calculus} 
\label{sec:Preliminaries}
 In this paper we consider infinite dimensional stochastic systems of the following form: 
 \begin{equation}\label{eq:Inf_SDE}
   \rd X = \calA X(t) \rd t+  F(X(t)) \rd t + G(X(t)) \rd W(t),  \quad X(0)= \xi
 \end{equation} 
 \noindent defined on the probability space $ (\Omega, \F, \Pb) $  with filtration $ \F_{t}, t \geq 0 $, for the time interval $t\in [0, T] $. Let, $H$ and $U$ Hilbert spaces, then  $\calA: D(\calA) \subset H \to H $ is a infinitesimal  generator , $ \xi $ is an  $ \F_{0}-$measurable $H-$valued random variable,   while    $  F : H \to H $ and $ G : U \to H  $ are nonlinear mappings that satisfy properly formulated Lipschitz and linear growth conditions (associated with the existence and uniqueness of solutions for infinite dimensional stochastic systems - see \cite[Theorem 7.2]{da1992stochastic}).  The term  $ W(t)\in U $ corresponds to a  $Q$-Wiener process that is defined based on the following proposition (see \cite[Chapter 4]{da1992stochastic}).  We use the notation $ X(\cdot, \omega) $  to denote a  state trajectory.
 
\begin{prop} Let $ \{e_{i}\}_{i=1}^{\infty} $  be a complete orthonormal system for the Hilbert Space $U$  such that $ Q e_{i}= \lambda_{i} e_{i}   $. Here, $ \lambda_{i} $ is the eigenvalue of $ Q\in L(U) $ that corresponds to  eigenvector $ e_{i} $, and $L(U)$ denotes the space of linear operators acting on $U$.  Then,  a   $Q$-Wiener process $ W(t) \in U$ satisfies the following properties:
 \begin{enumerate}
 \item  $ W $ is a Gaussian process on $ U $  with mean and variance: 
 \begin{equation}
 \Eb [W(t)] = 0, \quad   \Eb [W(t) W(t) ] = t Q, ~ t\geq 0. 
 \end{equation} 
 
 \item For arbitrary  $ t \geq 0 $, $ W $ has the following expansion:
 \begin{equation}\label{eq:Q_Wiener}
    W(t) = \sum_{j=1}^{\infty} \sqrt{\lambda_{j}} \beta_{j}(t) e_{j},
 \end{equation}
 where  $ \beta_{j}(t)  $  are real valued brownian motions that are mutually independent on $ (\Omega, \F, P). $
 \end{enumerate}
\end{prop}

 
 In this paper we will make use of Girsanov's theorem for systems evolving on Hilbert spaces. To this end, let us introduce the Hilbert space $U_{0}:=Q^{1/2}(U)\subset U$ with inner product: $\langle u, v\rangle_{U_{0}}:=\langle Q^{-1/2}u, Q^{-1/2}v\rangle_{U}$, $\forall u,v\in U_{0}$. The following proposition is from \cite[Theorem 10.18]{da1992stochastic}:

\begin{prop}[Girsanov] \label{girs} Let $\Omega$ be a sample space with a $\sigma$-algebra $\mathcal{F}$. Consider the following $H$-valued stochastic processes:
\begin{align}
\rd X&=(\calA X+F(X)) \rd t+G(X)\rd W(t),\quad X(0)=x \label{X}\\
\rd\tilde{X}&=(\calA \tilde{X}+F(\tilde{X}))\rd t+\tilde{B}(\tilde{X})\rd t+G(\tilde{X})\rd W(t),\quad\tilde{X}(0)=x, \label{X_tilde}
\end{align}
where $W\in U$ is a Q-Wiener process with respect to the measure $\mathbb{P}$. Moreover, $\forall\Gamma\in C([0,T]; H)$ let the {\it law} of $X$ defined as $\mathcal{L}(X(\cdot,\omega)\in\Gamma):=\mathbb{P}(\omega\in\Omega|X(\cdot,\omega)\in\Gamma)$ . Similarly, the law of $\tilde{X}$, is defined as $\tilde{\calL}(\tilde{X}(\cdot,\omega)\in\Gamma):=\mathbb{P}(\omega\in\Omega|\tilde{X}(\cdot,\omega)\in\Gamma)$. Then
\begin{equation}
\begin{split}
&\tilde{\calL}(\omega)= \mathbb{E}_{\mathbb{P}}\big[\exp\big(\int_{0}^{T}\langle\psi(s), dW(s)\rangle_{U_{0}}-\frac{1}{2}\int_{0}^{T}||\psi(s)||_{U_{0}}^{2}ds\big)|X(\cdot)\in\Gamma\big],
\end{split}
\end{equation}
where $\psi(t):=G^{-1}(X(t))\tilde{B}(X(t))\in U_{0}$. Here, we write for brevity $\tilde{\calL}(\omega)\equiv\tilde{\calL}(\tilde{X}(\cdot,\omega)\in\Gamma)$.
\end{prop}
\begin{proof}
Define the process:
\begin{equation}
\label{w_hat}
\hat{W}(t):=W(t)-\int_{0}^{t}\psi(s)\rd s.
\end{equation}
Based on \cite[Theorem 10.18]{da1992stochastic}, $\hat{W}$ is a Q-Wiener process with respect to a measure $\Qb$ determined by:
\begin{equation}
\label{girsanov_measure}
\begin{split}
\rd \Qb (\omega)&=\exp\big(\int_{0}^{T}\langle\psi(s),\rd W(s)\rangle_{U_{0}}-\frac{1}{2}\int_{0}^{T}||\psi(s)||_{U_{0}}^{2}\rd s\big)\rd\mathbb{P} \\ &=\exp\big(\int_{0}^{T}\langle\psi(s),\rd \hat{W}(s)\rangle_{U_{0}}+\frac{1}{2}\int_{0}^{T}||\psi(s)||_{U_{0}}^{2}\rd s\big)\rd\mathbb{P}.
\end{split}
\end{equation}
Now, using \eqref{w_hat}, eq. \eqref{X} is rewritten as:
\begin{equation}
\label{X_new}
\begin{split}
\rd X&=(AX+F(X))\rd t+G(X)\rd W(t) =(AX+F(X))\rd t+B(X)\rd t+G(X)\rd\hat{W}(t)
\end{split}
\end{equation}
Notice that  the above SPDE has the same form as \eqref{X_tilde}. Therefore, under the introduced measure $\Qb$, $X$ becomes equivalent to \eqref{X_tilde}. However, under the  measure $\mathbb{P}$, the SPDE in \eqref{X_new} behaves as the original system in \eqref{X}. In other words, eqs. \eqref{X} and \eqref{X_new} describe the same system on $(\Omega, \mathcal{F}, \mathbb{P})$. From the uniqueness of solutions and the aforementioned reasoning, one has
\[\Pb(\{\tilde{X}\in\Gamma\}) = \Qb(\{X\in\Gamma\}).\]
The result follows from \eqref{girsanov_measure}.
\end{proof}

To conclude this section, we note that when $\lambda_{j}=1$, $\forall j$, $W(t)$ corresponds to a cylindrical Wiener process (space-time white noise). In that case, the series in \eqref{eq:Q_Wiener} converges in another Hilbert space $U_{1}\supset U$, when the inclusion $\iota:U\rightarrow U_{1}$ is Hilbert-Schmidt. For more details see \cite{da1992stochastic}.

\section{Relative Entropy and Free Energy Dualities in Hilbert Spaces}
\label{sec:Free_Energy_Relative_Entropy}

In this section we provide the relation between free energy and relative entropy. The relation is valid for general probability measures including  measures defined on path spaces induced by infinite dimensional stochastic systems.  Here we will consider the general  measures $ \calL $ and $ \tilde{\calL} $.  
\begin{definition}{\label{def:Free_Energy} \textit{(Free Energy)} Let  $   \calL \in \Pc $ a probability measure and let  the function   $J \equiv J(X(\cdot, \omega) ): L^{p} \to  \Rb_{+}$ be a measurable function. Then the following term:
  \begin{equation}
V  =\frac{1}{\mu} \log_e   \int_{\Omega} \exp(\mu J )   \rd \calL (\omega),
  \end{equation}
 \noindent     is called the {\it free energy}\footnote{The function  $ \log_{e}  $ denotes the natural  logarithm.} of  $ J   $    with respect to $   \calL $ and  $ \mu \in \Rb$.}  
 \end{definition}  

  \begin{definition}{\label{def:Entropy} \textit{(Generalized Entropy)} Let  $   \calL  \in \Pc$  and  $  \tilde{\calL} \in \Pc$ then the relative entropy of $  \tilde{\calL} $  with respect to $   \calL $ is defined as: 
\[    S\left( \tilde{\calL }|| \calL  \ \right) = \left\{
\begin{array}{l l}
  -\int_{\Omega}   \frac{    \rd \tilde{\calL}(\omega) }{\rd \calL(\omega)}  \log_e  \frac{ \rd \tilde{\calL}(\omega)}{\rd \calL (\omega)}         \rd \calL(\omega),
  \mbox{if $\tilde{\calL}<<\calL $},  \\
  +\infty,  \quad \mbox{otherwise},\\ \end{array} \right. \]
where ``$<<$'' denotes absolute continuity of $\tilde{\calL}$ with respect to $\calL$ and
$\mathcal{L}_1$  denotes the space of Lebesgue measurable functions on  $ [0, \infty)$. We say that $ \tilde{\calL} $ is \textit{absolutely continuous} with respect to $ \calL $ and we write  $\tilde{\calL}<<\calL $ if  $ \calL(B) = 0 \Rightarrow \tilde{\calL}(B) = 0, ~ \forall B \in \F$. }
 \end{definition}  
 The  free energy and relative entropy relationship is expressed by the theorem that follows:
\begin{theorem} \textit{  Let  $( \Omega, {\F}) $  be a measurable space. Consider  $ \calL , \tilde{\calL}  \in  \Pc $ and    the definitions of free energy and relative entropy as expressed in  definitions \ref{def:Free_Energy}  and \ref{def:Entropy}.  Under the assumption that $\Qb_{\bbX}<<\Pb_{\bbX} $,  the following inequality holds:}
    \begin{align} \label{eq:Legendre}
    & - \frac{1}{\rho}   \log_e \Eb_{\calL} \bigg[ \exp( -\rho {J} )  \bigg]  \leq \bigg[    \Eb_{\tilde{\calL}}\left({J} \right)  -\frac{1}{\rho} S \left( \tilde{\calL} ||  \calL \right)  \bigg] ,
    \end{align}   
 \noindent  \textit{where $ \mathbb{E}_{\calL}, \mathbb{E}_{\tilde{\calL}}  $  is the expectation under the probability measure  $ \calL,\tilde{\calL} $  respectively and $ \rho \in \Rb_{+}$ and $ J :  L^{p} \to  \Rb_{+} $.  The inequality in \eqref{eq:Legendre} is the so called   Legendre Transform.   }  
\end{theorem}
 By defining the free energy  as temperature $  T = \frac{1}{\rho} $  the Legendre transformation has the form: 
\begin{equation}\label{eq:Legendre_StatMech}
V  \leq E  - T S,
\end{equation}
which has statistical mechanics interpretation.  The equilibrium probability measure has the classical form: 
\begin{equation}\label{eq:Gibbs}
\rd \calL^{*}(\omega) = \frac{\exp( - \rho J) \rd \calL(\omega)}{\int_{\Omega} \exp( - \rho J)  \rd \calL(\omega)},
\end{equation}
To verify that the measure in \eqref{eq:Gibbs} is the optimal measure it suffices to substitute \eqref{eq:Gibbs} in \eqref{eq:Legendre} and show that the inequality collapses to an equality \cite{entropy_2015}.   The statistical physics  interpretation of inequality \eqref{eq:Legendre_StatMech} is that, maximization of entropy results in reduction of the available energy. At the thermodynamic equilibrium the entropy reaches its maximum and the inequality collapses to equality. It can be shown that when the measures $ \tilde{\calL} $ and $ \calL $ are associated to paths generated by control and uncontrolled semi-linear SPDEs, then the free energy is value function that satisfies the HJB equation of an infinite dimensional stochastic optimal control problem. This observation motivates the use of  
\eqref{eq:Gibbs} for the development of stochastic control algorithms.    


\section{Variational Inference and Control in Hilbert Spaces}
\label{sec:Inference}

In this section we will derive our numerical algorithm for controlling stochastic infinite dimensional systems.To simplify our expressions, we will consider without loss of generality SPDEs with additive noise. Let the uncontrolled and controlled version of an $H$-valued process be given respectively by:
\begin{equation}\label{eq:SPDEs_NoControl}
\rd X(t) = (\A X   + F(X(t)))  \rd t +   \frac{1}{\sqrt{\rho}}  \rd W(t),~\text{and} ~ \rd \tilde{X} (t) = ( \A \tilde{X}   + F(\tilde{X}(t))  + \calU(t) ) \rd t  + \frac{1}{\sqrt{\rho}} \rd  W(t)
\end{equation}
 both with initial condition: $
X(0) =\tilde{X}(0)= \xi$. Here, $W\in U=H$ is a $Q$-Wiener process on $(\Omega, \mathcal{F}, \mathbb{P})$ with covariance operator $Q\in L(U)$. As in the previous section, the uncontrolled dynamics are equivalent to:
\begin{equation}\label{eq:SPDEs_NoControl_new}
\rd X(t) = (\A X   + F(X(t))+\calU(t))  \rd t +   \frac{1}{\sqrt{\rho}}  \rd \hat{W}(t),
\end{equation}
with respect to $\mathbb{P}$. Here, $\hat{W}$ is a $Q$-Wiener process with respect to another measure $\mathbb{Q}$. The law of the uncontrolled states, $\mathcal{L}(\cdot)$, defines a measure on the path space via \eqref{eq:SPDEs_NoControl} as  $\mathcal{L}(\omega):=\mathbb{P}(\omega|X(\cdot,\omega)\in\Gamma)$. Similarly, the law of controlled trajectories is $\tilde{\mathcal{L}}(\omega):=\mathbb{P}(\omega|\tilde{X}(\cdot,\omega)\in\Gamma)$. Finally, we suppose that there exists an optimal controller $\calU^{*}$ which corresponds to the law of optimal trajectories, $\mathcal{L}^{*}(\cdot)$.

In this section we derive controllers by formulating a new optimization problem in which we make use of the measure theoretic approach. We are looking for a control input $\calU(\cdot)$ that minimizes the distance to the optimal path law. That is:
 \begin{align}
    \calU^{*}(\cdot) =   \argmax_{\calU(\cdot)}  S(\mathcal{L}^{*}|| \tilde{\mathcal{L}}). 
  \end{align}
  Under the parameterization $  \calU = \calU(X(t);\vtheta)  $ the problem above will take the form:
     \begin{align*}
    \vtheta^{*}  =   \argmax   \bigg[ -  \int_{\Omega}   \frac{    \rd \calL^{*}(\omega) }{\rd  \tilde{\calL}(\omega)}  \log_e  \frac{ \rd \calL^{*}(\omega)}{\rd \tilde{\calL}(\omega)}  \rd  \tilde{\calL}(\omega) \bigg] 
           =   \argmin   \bigg[   \int_{\Omega}   \ \log_e   \frac{ \rd \calL^{*}(\omega)}{\rd \tilde{\calL}(\omega)}  \rd  \calL^{*}(\omega) \bigg].
  \end{align*}
To perform the optimization we will consider the chain rule property for the Radon-Nikodym derivative.For instance, this results in the following expression:
\begin{equation}
 \frac{ \rd \calL^{*}(\omega)}{\rd \tilde{\calL}(\omega)}   =  \frac{ \rd \calL^{*}(\omega)}{\rd \calL(\omega)}   \frac{ \rd \calL(\omega)}{\rd \tilde{\calL}(\omega)}. 
\end{equation}
 Note that the first derivative is given by \eqref{eq:Gibbs} while the second derivative is given by the change of measure between control and uncontrolled infinite dimensional stochastic dynamics.  Based on the discussion of the previous section, $\tilde{\mathcal{L}}(\omega)=\mathbb{Q}(\omega|X(\cdot,\omega)\in\Gamma)$ and $\mathcal{L}^{*}(\omega)=\mathbb{Q}^{*}(\omega|X(\cdot,\omega)\in\Gamma)$, where $\mathbb{Q}^*$ is properly defined. From  Proposition \ref{girs}, this  is computed by

\begin{equation}
\label{radon1}
\frac{ \rd \tilde{\calL}}{ \rd \calL } =  \frac{ \rd \mathbb{Q}}{ \rd \mathbb{P} }=\exp\bigg(\int_{0}^{T}\langle\psi(s), \rd W(s)\rangle_{U_{0}}-\frac{1}{2}\int_{0}^{T}||\psi(s)||_{U_{0}}^{2} \rd s\bigg),
\end{equation}
where $\psi(t):=\sqrt{\rho}\mathcal{U}(t)\in U$. In this paper we will parameterize our infinite dimensional control as follows:
\begin{equation}
\label{param_u}
\calU(t) = \sum_{\ell=1}^{N} m_{\ell} u_{\ell}(t)\in U\equiv H,
\end{equation}
so that
\begin{equation}\label{eq:ParameterControl1}
\calU(t)(\vx) = \sum_{\ell=1}^{N} m_{\ell}(\vx) u_{\ell}(t) =\vm(\vx)^{\rT} \vu(t)\in\mathbb{R}.
\end{equation}
Here, $ m_{\ell}\in U $ are design  functions that specify how the  actuation is incorporated into the infinite dimensional dynamical system. Under this parameterization, the change of measure between the two SPDEs  takes the form:
\begin{equation}
\label{radon_param}
\begin{split}
 \frac{ \rd \mathbb{Q}}{ \rd \mathbb{P} }=\exp\bigg(\sqrt{\rho}\int_{0}^{T}\vu(t)^{\top}\bar{ \vm}(t)-\rho\frac{1}{2}\int_{0}^{T}\vu(t)^{\top}\vM\vu(t)\rd t\bigg),
\end{split}
\end{equation}
where
\begin{equation}
\label{small_m}
\bar{ \vm}(t):=\bigg[\langle m_{1},\rd W(t)\rangle_{U_{0}},...,\langle m_{N},\rd W(t)\rangle_{U_{0}}\bigg]^{\top}\in\mathbb{R}^{N},
\end{equation}
\begin{equation}
\label{big_M}
\vM\in\mathbb{R}^{N\times N},\quad (\vM)_{ij}:=\langle m_{i},m_{j}\rangle_{U_{0}}.
\end{equation}

The following theorem provides the optimal control  $  \vu^{*} $ for the case of the controlled SPDEs of the form in \eqref{eq:SPDEs_NoControl}. 
\begin{lemma}\label{TheoremSOC} (Variational  Stochastic Control) Consider the controlled SPDE in \eqref{eq:SPDEs_NoControl} and let the following objective  function:
\begin{equation} 
   \vu^{*} =   \argmax  S(\calL^{*}|| \tilde{\calL})  
\end{equation}
The probability measure  $  \calL^{*} $ is  induced by the optimally controlled SPDE   in \eqref{eq:SPDEs_NoControl} and has the  form:
\begin{equation}\label{eq:optimal_measure}
\rd \calL^{*}(\omega) = \frac{\exp( - \rho J(Xq)) \rd \calL(\omega)}{\int_{\Omega} \exp( - \rho J(X))  \rd \calL(\omega)},
\end{equation}
The probability measure  $ \tilde{\calL}$ is induced by controlled trajectories of the SPDEs  when infinite dimensional control  $ \calU(t) $  is determined by \eqref{eq:ParameterControl2} and $  \vu(t) $  in  \eqref{eq:ParameterControl1}  is  parameterized as  follows: 
\begin{equation} \label{eq:ParameterControl2}
 \vu(t) =
  \vu_{i}\equiv \vu(t_i)    \quad  \text{if            } \quad i \Delta t \le t < (i+1)\Delta t ,\quad\forall t\in[0,T]
\end{equation}
\noindent with $i = \{0, 1, \hdots L\}$.  Under the aforementioned representation, the optimal control  is provided by the following expression:
\begin{equation}\label{eq:optimalcontrol}
\vu^{*}_{i} =      \frac{1}{\sqrt{\rho} \Delta t}   \vM^{-1}  \Eb_{\calL}   \bigg[   \frac{\exp( - \rho J) }{  \Eb_{\calL} [ \exp( - \rho J)] }  \delta \vu_{i} \bigg], \quad \text{and} \quad 
 \delta \vu_{i} :=  \int_{t_i}^{t_{i+1}}\bar{ \vm}(t).
 \end{equation}
\end{lemma} 
\begin{proof}
 Under the parameterization $  \calU(\vx,t) =  \vm(\vx)^{\rT} \vu(t)  $ the problem above will take the form:
     \begin{align*}
    \vu^{*}  =   \argmin   \bigg[   \int_{\Omega}   \ \log_e  \frac{ \rd \calL^{*}(\omega)}{\rd \tilde{\calL}(\omega)}  \rd \calL^{*}(\omega) \bigg]  &=    \argmin   \bigg[   \int_{\Omega}   \ \log_e  \frac{ \rd \calL^{*}(\omega)}{\rd \calL(\omega)} \frac{ \rd \calL(\omega)}{\rd \tilde{\calL}(\omega)}  \rd \calL^{*}(\omega) \bigg]. 
  \end{align*} 
By using \eqref{radon_param} minimization of the last expression is equivalent to the minimization of the  expression: 
     \begin{align*}
 \Eb_{\calL^{*}}  \bigg[     \log_{e} \frac{ \rd \calL(\omega)}{\rd \tilde{\calL}(\omega)}  \bigg]   &=   -\sqrt{\rho} \Eb_{\calL^{*}}  \bigg[  \int_{0}^{T}\vu(t)^{\top}\bar{ \vm}(t)\bigg]+\frac{1}{2} \rho\Eb_{\calL^{*}}  \bigg[  \int_{0}^{T}\vu(t)^{\top}\vM\vu(t)\rd t\bigg].
  \end{align*} 
 The goal is to find the function $\vu^*(\cdot)$ which minimizes. However, since we inevitably apply the control in discrete time it suffices to consider the class of step functions:
    \begin{align*}
 & \Eb_{\calL^{*}}  \bigg[     \log_{e} \frac{ \rd \calL(\omega)}{\rd \tilde{\calL}(\omega)}  \bigg]  =  -\sqrt{\rho} \sum_{i=0}^{L-1}\vu_{i}^{\top}\Eb_{\calL^{*}}  \bigg[  \int_{t_{i}}^{t_{i+1}}\bar{ \vm}(t)\bigg]+
\frac{1}{2}\rho\sum_{i=0}^{L-1}\vu_{i}^{\top}\vM\vu_{i}\Delta t,
  \end{align*} 
 where we have used the fact that $\vM$ is symmetric and constant with respect to time. Minimization of the expression above with respect to  $ \vu_{i}$ results in: 
  \begin{equation}
      \vu_{i}^{*} =  \frac{1}{\sqrt{\rho} \Delta t  }\vM 
       ^{-1}  \Eb_{\calL^{*}}  \bigg[\int_{t_{i}}^{t_{i+1}}\bar{ \vm}(t)\bigg].
  \end{equation}
 Since we cannot sample from the $ \calL^{*}$, we need to
change the expectation to be an expectation with respect to
the uncontrolled dynamics,  $ \calL$. We can then directly sample
trajectories from $ \calL $ to approximate the controls. The change
in expectation is achieved by applying the Radon-Nikodym
derivative. The result is equation \eqref{eq:optimalcontrol}. 
 \end{proof}

\subsection{Iterative  Control of SPDEs}
\label{sec:Iterative}
     We derive an iterative scheme  that can be used for  stochastic optimization and  be  implemented in a receding horizon  fashion.   In particular, let us  consider the controlled dynamics at iteration $ i^{th}$  given by: 
\begin{equation}\label{eq:SPDEs_Control_Iter_i}
\rd X^{(i)}(t) = ( \A X^{(i)}   + F(X^{(i)})  + \calU^{(i)}(t) ) \rd t  + \frac{1}{\sqrt{\rho}} \rd  W(t),
\end{equation}
where $ \calU^{i}(t) $  is the control at the $ i^{th} $  iteration.  As we have already shown, the uncontrolled dynamics can be equivalently written as:
\begin{equation}\label{eq:SPDEs_NoControl_new}
\begin{split}
\rd X(t)& = (\A X   + F(X(t)))  \rd t +   \frac{1}{\sqrt{\rho}}  \rd W(t) =(\A h   + F(X(t))+\calU^{(i)}(t) )  \rd t +   \frac{1}{\sqrt{\rho}}  \rd W^{(i)}(t),
\end{split}
\end{equation}
where $W^{(i)}$ is a $Q$-Wiener process with respect to some measure $\mathbb{Q}^{(i)}$ with:
\begin{equation}
\label{w_hat_i}
W^{(i)}(t):=W(t)-\int_{0}^{t}\rho\calU^{(i)}(s)\rd s,
\end{equation}
Again here we define the path measure $\calL^{i}(\omega):=\mathbb{P}(\omega|X^{(i)}(\cdot,\omega)\in\Gamma)  $ induced by  \ref{eq:SPDEs_Control_Iter_i}  and the path measure $\calL(\omega):=\mathbb{P}(\omega|X(\cdot,\omega)\in\Gamma)  $ induced by \eqref{eq:SPDEs_NoControl_new}. Then according to \ref{girs} we have:   
\begin{equation}
\label{radon_i}
\begin{split}
\frac{ \rd \mathbb{\calL}^{(i)}}{ \rd \calL } = \frac{ \rd \mathbb{Q}^{(i)}}{ \rd \mathbb{P} }=\exp\big(\sqrt{\rho}\sum_{k=0}^{L-1}\vu_k^{(i)\top}\int_{t_k}^{t_{k+1}}\bar{ \vm}^{(i)}(t)+\rho\frac{1}{2}\sum_{k=0}^{L-1}\vu_k^{(i)\top}\vM\vu_k^{(i)}\Delta t\big),
\end{split}
\end{equation}
where
\begin{equation}
\label{small_m}
\mathbb{R}^{N}\ni\bar{ \vm}^{(i)}(t):=\bigg[\langle m_{1},\rd W^{(i)}(t)\rangle_{U_{0}},...,\langle m_{N},\rd W^{(i)}(t)\rangle_{U_{0}}\bigg]^{\top},
\end{equation}

 \begin{lemma} (Iterative Stochastic Control) Consider the controlled SPDE in \eqref{eq:SPDEs_NoControl} and the parameterization of the control as specified by \eqref{eq:ParameterControl1} and  \eqref{eq:ParameterControl2}. The iterative control scheme is given by the following expression:
\begin{equation}\label{eq:Iterative_optimalvariation1}
\vu^{*}_{j}{}^{(i+1)} =  \vu^{*}_{j}{}^{(i)}   +   \frac{1}{\sqrt{\rho} \Delta t}   \vM^{-1}  \Eb_{\calL^{(i)}}   \bigg[   \frac{\exp( - \rho \tilde{J}) }{  \Eb_{\calL^{(i)}} [\exp( - \rho  \tilde{J}] }  \delta \vu^{(i)}_{j}    \bigg],\quad and \quad 
 \delta \vu^{(i)}_j =   \int_{t_j}^{t_{j+1}}\bar{ \vm}^{(i)}(t). 
\end{equation}
and
 \begin{equation} \tilde{J} = J + \zeta, 
 \end{equation}
with the control path dependent  function   $ \zeta^{(i)} = \zeta(\calU^{(i)}): [0,T] \times \calO \to \Rb  $ defined as follows:  
\begin{align}\label{eq:zeta}
\zeta(\calU^{(i)}) = &\frac{1}{\sqrt{\rho}}\sum_{k=0}^{L-1}\vu_k^{(i)\top}\int_{t_k}^{t_{k+1}}\bar{ \vm}^{(i)}(t)+\frac{1}{2}\sum_{k=0}^{L-1}\vu_k^{(i)\top}\vM\vu_k^{(i)}\Delta t.
\end{align}
The expectation in \eqref{eq:Iterative_optimalvariation1} is taken with respect to the probability path  measure $ \calL^{(i)} $  induced  by sampled  trajectories generated using \eqref{eq:SPDEs_Control_Iter_i}.
\end{lemma} 
\begin{proof} In order to derive the iterative scheme, we perform one step of importance sampling. In particular,  instead of sampling from the uncontrolled SPDE \eqref{eq:SPDEs_NoControl}  to evaluate  
the expectation in   \eqref{eq:optimalcontrol}  we sample  using the  controlled SPDE  \eqref{eq:SPDEs_Control_Iter_i}. In addition, we  modify   \eqref{eq:optimalcontrol}  so that to   perform the appropriate change of measure between the uncontrolled version of infinite dimensional dynamics and the controlled version at iteration  $ i $.
Next we modify  equations    \eqref{eq:optimalcontrol}  by considering \eqref{radon_i} and \eqref{w_hat_i}.    
\begin{equation}\label{eq:Iterative_optimalvariation}
\vu^{(i+1)}_{j} =      \frac{1}{\sqrt{\rho} \Delta t}   \vM^{-1}  \Eb_{\calL^{(i)}}   \bigg[ \frac{\rd \calL}{\rd  \calL^{(i)}}  \frac{\exp( - \rho J) }{  \Eb_{\calL} [ \exp( - \rho J)] } \delta \vu_{j}    \bigg], 
\end{equation}
Regarding $\delta \vu_{j}$, one has:
\begin{equation*}
\begin{split}
&\bigg(\int_{t_j}^{t_{j+1}}\bar{ \vm}(t)\bigg)_{l}=\int_{t_j}^{t_{j+1}}\langle m_l,\rd W(t)\rangle_{U_{0}}=\\
&\int_{t_j}^{t_{j+1}}\langle m_l,\rd W^{(i)}(t)+\sqrt{\rho}\calU^{(i)}(t)\rd t\rangle_{U_{0}}= \int_{t_j}^{t_{j+1}}\langle m_l,\rd W^{(i)}(t)\rangle_{U_{0}}+\sqrt{\rho}\bigg[\langle m_{l},m_{1}\rangle_{U_{0}},...,\langle m_{l},m_{N}\rangle_{U_{0}}\bigg]\vu_{j}^{(i)}\Delta t.
\end{split}
\end{equation*}
It follows that:
\[\int_{t_j}^{t_{j+1}}\bar{ \vm}(t)= \int_{t_j}^{t_{j+1}}\bar{ \vm}^{(i)}(t)+\sqrt{\rho}\Delta t\vM \vu_{j}^{(i)}. \]

Substitution of  the Radon-Nikodym derivative yields the final result in \eqref{eq:Iterative_optimalvariation1}. Note that under $\mathbb{Q}^{(i)}$ renders $W^{(i)}$ a standard $Q$-Wiener process.
\end{proof}

For the purposes of implementation we will approximate the optimal controls \eqref{eq:Iterative_optimalvariation1} as:
 \begin{equation}\label{eq:optimalvariation_1}
( \delta \tilde{\vu}^{(i)}_{j} )_{l}:=   \sum_{s=1}^{R}\langle m_{l},\sqrt{\lambda}_{s}e_{s}\rangle_{U_{0}}\Delta\beta^{(i)}_{s}(t_j),
 \end{equation}
where $\Delta\beta^{(i)}_{s}(t_j)\sim\mathcal{N}(0, \Delta t)$ under $\mathbb{Q}^{(i)}$. 
 Next we discuss the application of the iterative stochastic control on two examples of SPDEs. 


  
 

  \section{Experiments}
  \label{sec:Experiments}
  In this section, we  present simulation  results on two infinite dimensional stochastic systems. The first systems is the stochastic Heat equation  and the second system is the Nagumo SPDE. The iterative stochastic optimal control is used for open loop  trajectory  optimization and for MPC. \\ 
\textbf{Heat SPDE:} The 1-D  stochastic heat equation with homogeneous Dirichlet boundary conditions can be used to simulate the diffusion of heat along a rod insulated on the sides and exposed to freezing conditions at the end points. Our experiments consisted of achieving desired temperature levels at specific positions along a rod in the presence of space-time stochastic disturbing forces. As seen in Fig. \ref{fig:heat_eqn}, the MPC  has robust performance compared to open-loop controller with the mean temperature profile closer to the desired temperature levels and tighter sigma bounds in the presence of space-time white noise. 
\begin{figure}[!htb]
\minipage{0.3\textwidth}
  \includegraphics[width=\linewidth]{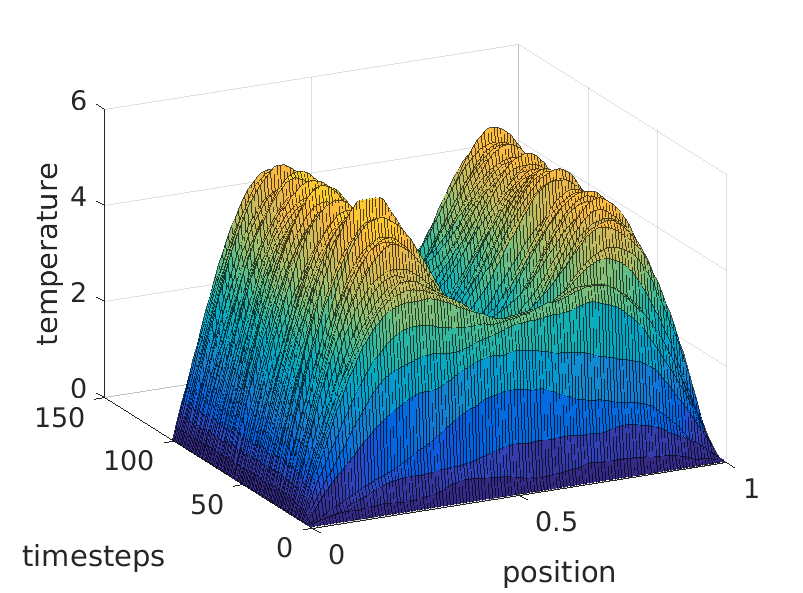}
\endminipage\hfill
\minipage{0.3\textwidth}
  \includegraphics[width=\linewidth]{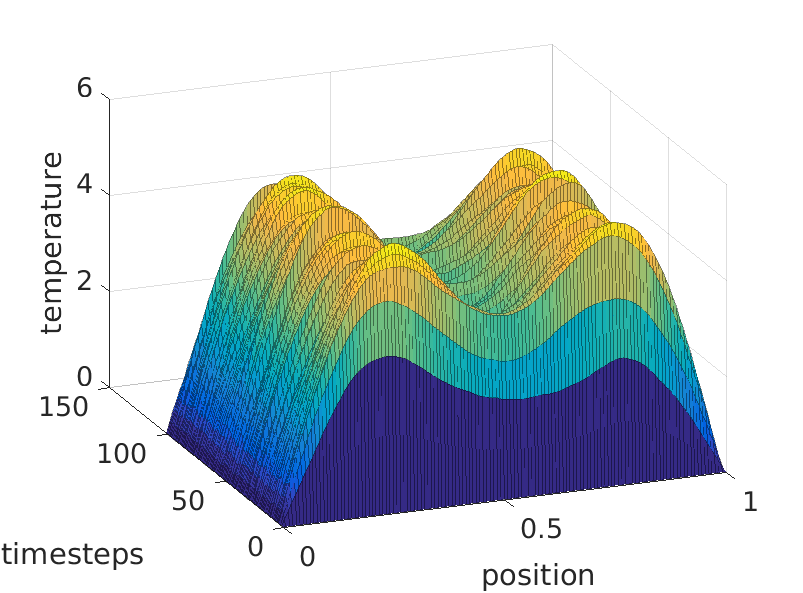}
\endminipage\hfill
\minipage{0.40\textwidth}%
  \includegraphics[width=\linewidth]{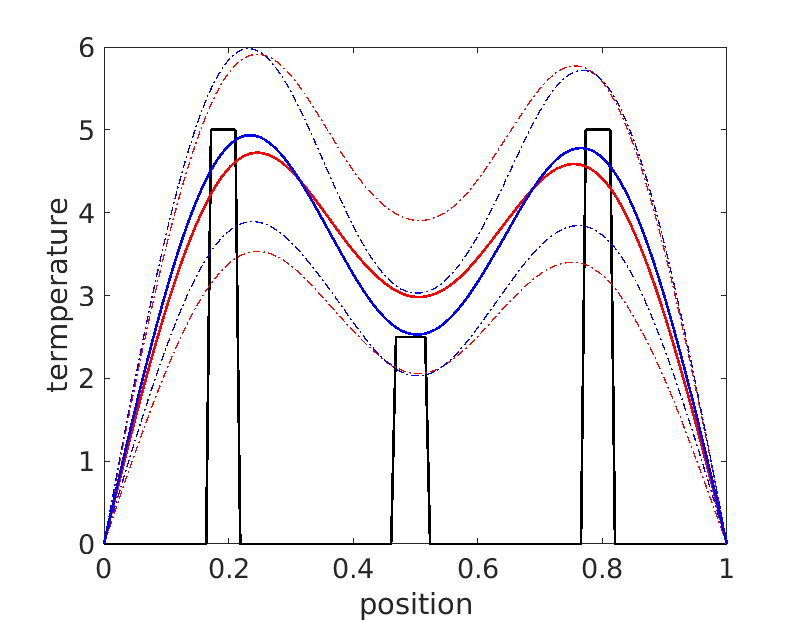}
\endminipage
\caption{Plots showing the evolution of temperature profile along a rod for Model Predictive Controller \textit{(left)}, Open-Loop controller (\textit{middle}) and the mean temperature profile for both controllers over an episode within $\pm$1-$\sigma$ bounds \textit{(right)}. Curves in red record performance of open-loop controller, those in blue record MPC. Black bars indicate desired temperature levels.}
\label{fig:heat_eqn}
\end{figure}

\textbf{Nagumo SPDE:} The stochastic Nagumo equation with homogeneous Neumann boundary conditions is a reduced model for wave propagation of the voltage $u$ in the axon of a neuron \cite{lord_powell_shardlow_2014}. The Nagumo equation is expressed as follows: 
\begin{align*}
& u_t = \epsilon u_{xx} + u(1-u)(u-\alpha) + \sigma dW(t), \quad u_x(t,0) = u_x(t,a) = 0, u(0,x) = (1+\exp(-(2-x)/\sqrt[]{2}))^{-1}
\end{align*}
The parameter $\alpha$ determines the speed of a wave traveling down the length of the axon and $\epsilon$ the rate of diffusion. From simulating the deterministic version of the above pde for $a=10,\,\epsilon=1$ and $\alpha=-0.5$, we observed that it requires about 10 seconds for the wave to propagate to the end of the axon 
 An open-loop infinite-dimensional controller was employed to accelerate the propagation of the voltage and to suppress the propagation of the voltage in about 2.5 seconds. The plots shown in the figure below demonstrate the achievement of desired behavior in the axon. 
\begin{figure}[!htb]
\minipage{0.5\textwidth}
  \includegraphics[width=\linewidth]{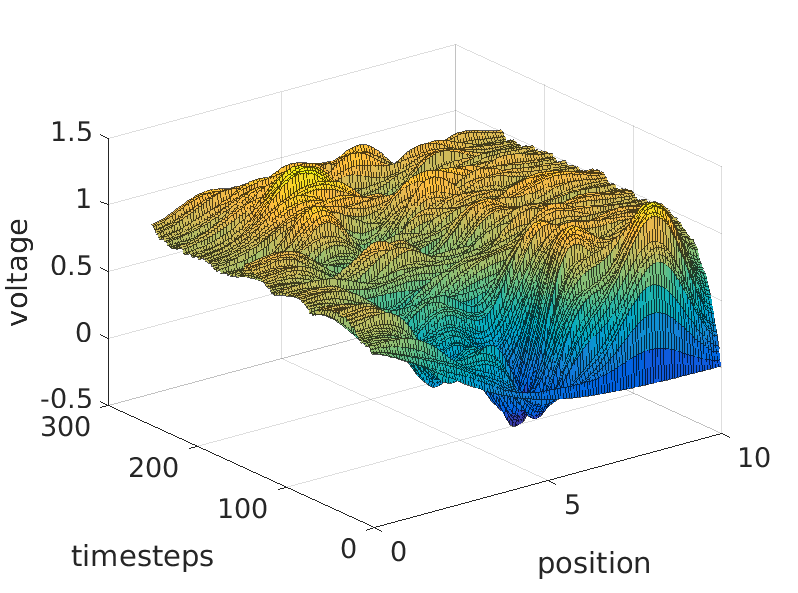}
\endminipage\hfill \quad
\minipage{0.5\textwidth}%
  \includegraphics[width=\linewidth]{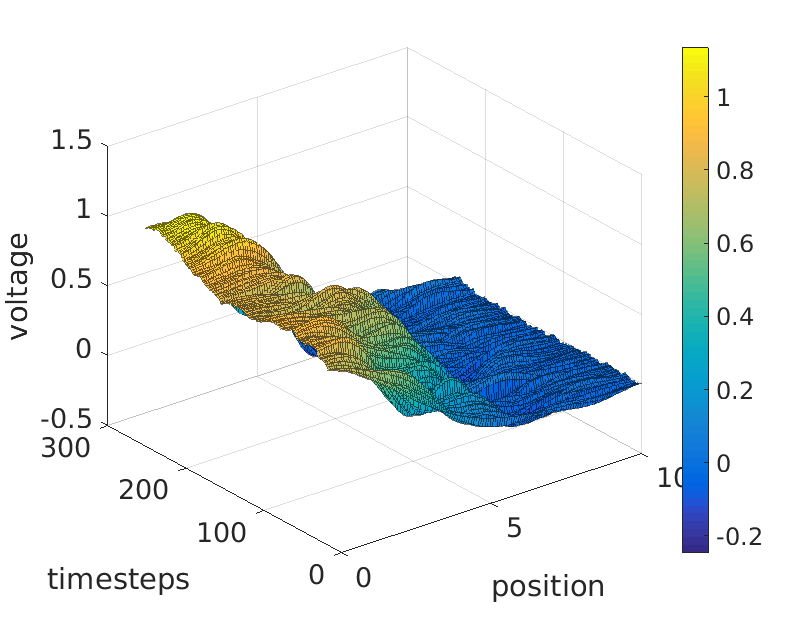}
\endminipage
\caption{Plots showing the acceleration of voltage propagation in an axon \textit{(left)} and suppression of voltage propagation in an axon \textit{(right)}}.
\label{fig:nagumo}
\end{figure}
  \vspace{-0.6cm}
\section{Conclusions}
\label{sec:conclusions}
 \vspace{-0.2cm}
  We  present  an information theoretic formulation for stochastic  optimal control of infinite dimensional dynamical systems.   The analysis relies on concepts drawn from the theory of stochastic calculus in Hilbert spaces,  the   relative entropy and free energy  relation and its connections to stochastic dynamic programming.  The  resulting algorithm can be used for stochastic trajectory optimization and MPC for a   large class of systems with dynamics governed by SPDEs. The work in this paper is a generalization of the path integral and information theoretic control to infinite dimensional spaces and is a  significant step towards the development of scalable and real time control algorithms for infinite dimensional stochastic  systems. Future directions involve, the theoretical analysis of the convergence,  application to higher order infinite dimensional systems, fully nonlinear SPDEs and application to real systems.

\bibliographystyle{unsrtnat}
\bibliography{ref.bib}

\end{document}